\documentclass{amsart}

\newtheorem{theorem}{Theorem}[section]

\theoremstyle{definition}
\theoremstyle{pro}
\newtheorem{definition}[theorem]{Definition}
\newtheorem{pro}[theorem]{Proposition}

\usepackage{tikz}
\usepackage[small,nohug,heads=vee]{diagrams}
\usepackage{mathrsfs}
\diagramstyle[labelstyle=\scriptstyle]
\theoremstyle{remark}
\newtheorem{remark}[theorem]{Remark}
\numberwithin{equation}{section}

\begin{document}
\title{ Index theorem for $\mathbb{Z}/2$-harmonic spinors}

\author{Ryosuke Takahashi}
\address{Institute of Mathematical Scences, Chinese University of Hong Kong,
Academic Building No.1, CUHK, Sha Tin, New Territories, HK}
\email{rtakahashi@ims.cuhk.edu.hk}


%    General info

\date{}


\begin{abstract}
Let $M$ denote a compact 3-manifold. The author proved in [8] that there exists a Kuranishi structure for the moduli space of pairs consisting of a Riemannian metric on $M$ and a non-zero $\mathbb{Z}/2$-harmonic spinor subject to certain natural regularity assumptions. This paper proves that the virtual dimension of $\mathbb{Z}/2$-harmonic spinors for a generic metric is equal to zero. The paper also
computes the virtual dimension of certain $\mathbb{Z}/2$-harmonic spinors on 4-manifolds using an index theorem developed by Jochen Bruning and Robert Seeley and, independently, Fangyun Yang.
\end{abstract}
\maketitle
\section{Introduction and Main theorem}
Let $M$ be a closed oriented smooth 3-manifold. We define the following spaces
\begin{align*}
\mathcal{X}&=\{\mbox{ Riemannian metric defined on }M\mbox{ }\},\\
\mathcal{A}&=\{ \mbox{ }C^1\mbox{-embeddeding }S^1\rightarrow M\mbox{ }\},\\
\mathcal{Y}&=\mathcal{X}\times\mathcal{A}.
\end{align*}
For any $(g,\Sigma)\in \mathcal{Y}$, we choose a spinor bundle, $\mathcal{S}_{g,\Sigma}$, defined on $M-\Sigma$ which cannot be extended on $M$. The choice of this spinor bundle is not unique, but there are only finitely many choices. We fix a choice of $\mathcal{S}_{g,\Sigma}$ throughout the rest of this paper.\\

The author introduced in [8] the space $\mathfrak{M}$ which consists of $(g,\Sigma,\psi)\in \mathcal{Y}\times L^2_1(\mathcal{S}_{g,\Sigma})$ that satisfy the following
conditions: First, $\psi$ obeys the Dirac equation defined by the metric $g$ on $M -\Sigma$. Second, $|\psi|$
can be extended to the whole of $M$ as a H\"older continuous function. Third,
$\frac{|\psi|(p)}{\mbox{dist}(p,\Sigma)^{\frac{1}{2}}}$ is
bounded away from zero near $\Sigma$. The space $\mathfrak{M}$ is said here to be the moduli space of $\mathbb{Z}/2$-harmonic spinors on $M$.\\

The notion of a $\mathbb{Z}/2$-harmonic spinor was introduced by Taubes [1], [3] to
describe the behavior of certain non-convergent sequences of $PSL(2;\mathbb{C})$-connections
on 3-manifolds. This notion appeared again in Haydys and Walpuski's analysis of noncompact
sequences of solutions to multi-spinor generalizations of the Seiberg-Witten
equations on 3-manifolds [6]. Analogous $\mathbb{Z}/2$-harmonic spinors on 4-dimensional manifolds
appeared in the work by Taubes on the behavior of non-compact sequences of solutions to the
Kapustin-Witten equations [2], to the multi-spinor Seiberg-Witten equations on 4-manifolds
[4], and to the Vafa-Witten equations [5]. All of these equations have potentially important
applications. For example, Haydys and Walpuski [13] [14],  conjecture a fundamental relation between the multi-spinor Seiberg-Witten equations on 3-manifolds and the spaces of $G_2$-instantons on certain 7-dimensional manifolds (also see [7]). Meanwhile, Witten has conjectured [15] that spaces of solutions of the Kapustin-Witten equations can be used to compute the Jones’ polynomial for knots in $S^3$. All of these applications require some understanding of the behavior of non-convergent sequences of solutions to the relevant equations. What has been shown by Taubes and Haydys-Walpuski is that limits of
non-convergent sequences of solutions to the relevant equations can be defined (after a
renormalization) on the complement of a closed set in the ambient manifold of Hausdorff
dimension at most 2 that contains a dense, open $C^1$-submanifold. This bad set is, in all cases,
the zero locus of a $\mathbb{Z}/2$-harmonic spinor. This being the case, one must come to terms with
$\mathbb{Z}/2$-harmonic spinors and their zero locus. This paper
and [8] are the first steps to this end. Here, as in [8], some additional regularity is assumed −
that the zero locus of the $\mathbb{Z}/2$-harmonic spinor is everywhere a codimension 2-submanifold. Thus, it is assumed to be a union of embedded circles in case when $M$ has dimension 3, and
an embedded surface in the dimension 4 case.\\

Suppose now that $M$ is a closed, oriented 3-manifold. The main structure theorem for
 $\mathfrak{M}$ is as follows (see [8]):
\begin{theorem}
Let $p=(g, \Sigma, \psi)\in\mathfrak{M}$. There are
\begin{align*}
&a).\mbox{ }\mbox{two finite dimensional vector spaces } \mathbb{K}_0,\mathbb{K}_1,\mbox{ a ball } \mathbb{O}_0\subset\mathbb{K}_0 \mbox{ centered at }0,\\
&b).\mbox{  a set } \mathcal{B}\subset \mathcal{X}\mbox{ with } \mathcal{B} = p_1(\mathcal{N}) \mbox{ being the projection of }\mathcal{N},\mbox{ a neighborhood of }p,\\
&\mbox{ }\mbox{ }\mbox{ }\mbox{ from }\mathcal{Y}\mbox{ to
}\mathcal{X }, \mbox{ and}\\
&c).\mbox{ } f:\mathcal{B}\times \mathbb{O}_0\rightarrow \mathbb{K}_1\mbox{ a }C^1\mbox{-map in the sense of Frechet differentiation},
\end{align*}
such that $f^{-1}(0)$ is homeomorphic to a neighborhood of $p$ in $\mathfrak{M}$.
\end{theorem}
This theorem implies the following: The subset in $\mathfrak{M}$ with a fixed metric component, say $g=g_0$, is a finite dimensional object. This fixed metric subset is denoted henceforth as $\mathfrak{M}_{g_0}$.
The virtual dimension of
$\mathfrak{M}_{g_0}$
is defined as follows: Let $\mathbb{K}_0$ and $\mathbb{K}_1$ denote the vector spaces
in the $g_0$ version of Theorem 1.1. The virtual dimension of
$\mathfrak{M}_{g_0}$
is $dim(\mathbb{K}_0) - dim(\mathbb{K}_1)$. The following is one of the main results of this paper.

\begin{theorem}
Let $(g_0, \Sigma_0, \psi_0)$ be a point in $\mathfrak{M}$. Then the corresponding vector spaces $\mathbb{K}_0$ and $\mathbb{K}_1$ from Theorem 1.1 have the same dimension. In particular, the virtual dimension of $\mathfrak{M}_{g_0}$ is zero.
\end{theorem}

\begin{remark}Note that if $(g, \Sigma, \psi)$ is in $\mathfrak{M}$, then so is $(g, \Sigma, c\psi)$ with $c$ being any non-zero complex number. This in turn implies that the set of $(g, \Sigma, c\psi)$ from $\mathfrak{M}$ with $\psi$ having $L^2$ norm equal to 1 has formal dimension -1. This last observation supports a conjecture made by Haydys and Walpuski [13] with regards to the multi-spinor Seiberg-Witten equations.
\end{remark}

The proof of Theorem 1.2 occupies the first part of this paper. The second part of this
paper considers a generalization of Theorem 1.2 to the case when $M$ is a closed, oriented
manifold of dimension 4. This part considers an analog of $\mathfrak{M}$, $\mathfrak{M}_{T^2}$, consisting of triples $(g, \Sigma, \psi)$
where $g$ is a Riemannian metric, $\Sigma$ is a $C^1$ embedded 2-dimensional torus in $M$ with trivial
normal bundle and $\psi$ is a harmonic, self-dual spinor on the complement of $\Sigma$ (defined by a
Spin structure on the complement of $\Sigma$) whose norm extends across $\Sigma$ as a H\"older continuous
function vanishing on $\Sigma$ and obeying
$\frac{|\psi|(p)}{\mbox{dist}(p,\Sigma)^{\frac{1}{2}}}$ on a neighborhood of $\Sigma$. Note that in the
context of [4], [2] and [5], there is no a priori reason why the zero locus of $|\psi|$
should be a torus (and with trivial normal bundle too) even in the event that it is a $C^1$ submanifold. This constraint on the topology is an extra condition that is imposed here. In any event, even with the torus restriction, the analog of Theorem 1.1 for this 4-dimensional version of $M$ has yet to be proved. However, assuming that Theorem 1.1 holds for a given triple $(g_0, \Sigma_0, \psi_0)$ as just described, then the difference between the dimensions of the associated spaces $\mathbb{K}_0$ and $\mathbb{K}_1$ can be viewed as a virtual dimension for $\mathfrak{M}_{T^2}$ near $(g_0, \Sigma_0, \psi_0)$.\\

 Even though we don't have the 4-dimensional version of Theorem 1.1 and Fredholm property for the linearization of $\mathfrak{M}_{T^2}$ at $p=(g_0, \Sigma_0, \psi_0)$, denoted by $F_{T^2,p}: \mathbb{K}_0\rightarrow \mathbb{K}_1$ (defined in (4.6), section 4.2), the corresponding index can still be obtained based on the 3-dimensional linearization argument. $F_{T^2,p}$ is determined by the leading term of $\psi_0$, denoted by $d^{\pm}$. In the proof of Theorem 1.1 [8], it is true that the Fredholmness holds for any $d^{\pm}$ satisfying $|d^+|^2+|d^-|^2>0$. By assuming the same Fredholm property for $F_{T^2,p}$, i.e., $F_{T^2,p}$ is Fredholm for any $d^{\pm}$ satisfying $|d^+|^2+|d^-|^2>0$, we will be able to prove a four-dimensional
version of Theorem 1.2. This will be Theorem 4.4 in Section 4.

\section{Preliminary: Linearization of $\mathfrak{M}$}
\subsection{Some background properties and notations}
In this subsection, we will introduce some notations and propositions that will be needed in the proof of  Theorem 1.2. Then, we will also briefly go through the linearization argument of $\mathfrak{M}$ which appears in [8]. After having explained this argument, we can then define $\mathbb{K}_1$ and $\mathbb{K}_0$ precisely. We will omit all proofs of these propositions because they are all in [8].\\

First of all, for any $(g,\Sigma)\in \mathcal{Y}$, we can parametrize a small tubular neighborhood of $\Sigma$, $N$, by $(t,z)\in S^1\times\mathbb{D}_R$ where $\mathbb{D}_R$ is a complex disc of radius $R>0$ in $\mathbb{C}$. In addition, a $\mathbb{Z}/2$ spinor bundle can be written as $\mathcal{S}_{g,\Sigma}=\mathcal{S}_g\otimes \mathcal{I}_{\Sigma}$ where $\mathcal{S}_g$ is the spinor bundle defined on $M$ and $\mathcal{I}_{\Sigma}$ is a non-extendable real line bundle over $M-\Sigma$. On $N-\Sigma \simeq T^2\times (0,R)$, we have $\mathcal{S}_{g,\Sigma}=(\mathcal{S}_{S^1}\otimes\mathcal{I}_{\Sigma}) \oplus (\mathcal{S}_{S^1}\otimes\mathcal{I}_{\Sigma})$ where $\mathcal{S}_{S^1}\simeq \pi^*(\mathcal{S})$ is the pull-back bundle of the spinor bundle from the map $\pi: N-\Sigma \rightarrow S^1$ (by sending $(t,z)$ to $t$). The detail of this argument can be found in Section 2.1 of [9] or Appendix B in [8].\\

The following proposition can be found in Section 3.1 of [8].
\begin{pro} \ \\
$\mbox{a).}$ $ L^2(M-\Sigma; \mathcal{S}_{g,\Sigma})=ker(D|_{L^2})\oplus range(D|_{L^2_1})$,\\
$\mbox{b).}$ For any $v\in ker(D|_{L^2})$, we have
\begin{align*}
v=\left(\begin{array}{c}
\frac{c^+(t)}{\sqrt{z}}\\
\frac{c^-(t)}{\sqrt{\bar{z}}}
\end{array}\right)+v_{R},
\end{align*}
$\mbox{ }\mbox{ }\mbox{ }$    where $v_{R}=O(|z|^{\alpha})$ for some $\alpha > 0$ and $c^+,c^-\in C^{\infty}(N-\Sigma, \mathcal{S}_{S^1}\otimes\mathcal{I}_{\Sigma})$.\\
$\mbox{c).}$ For any $u\in ker(D|_{L^2_1})$, we have
\begin{align*}
u=\left(\begin{array}{c}
 d^+(t)\sqrt{z}\\
 d^-(t)\sqrt{\bar{z}}
\end{array}\right)+u_R.
\end{align*}
$\mbox{ }\mbox{ }\mbox{ }$    where $u_{R}=O(|z|^{\beta})$ for some $\beta >\frac{1}{2}$  and $d^+,d^-\in C^{\infty}(N-\Sigma, \mathcal{S}_{S^1}\otimes\mathcal{I}_{\Sigma})$.\\
$\mbox{ }\mbox{ }\mbox{ }$     Moreover, $u_R$ is in $L^2_2(N-\Sigma;\mathcal{S}_{g,\Sigma})$.
\end{pro}
Suppose that $\mathcal{S}_{S^1}|_{N}$ is a trivial complex line bundle. Using the notation from Proposition
2.1 b), we define the map
\begin{align*}
B:ker(D|_{L^2})\rightarrow L^2(S^1;\mathbb{C}^2)
\end{align*}
by sending $v$ to $(c^+,c^-)$. Now, $L^2(S^1;\mathbb{C}^2)$ can be decomposed in the following way:
\begin{align*}
Exp^+=\bigg\{\left(\begin{array}{c}
\sum_{l\in \mathbb{Z}} p_l e^{ilt}\\
\sum_{l\in \mathbb{Z}} -sign(l)ip_l e^{ilt}
\end{array}\right)\bigg| (p_l)\in l^2\bigg\},\\
Exp^-=\bigg\{\left(\begin{array}{c}
\sum_{l\in \mathbb{Z}} p_l e^{ilt}\\
\sum_{l\in \mathbb{Z}} sign(l)ip_l e^{ilt}
\end{array}\right)\bigg| (p_l)\in l^2\bigg\}.
\end{align*}
Then $L^2(S^1;\mathbb{C}^2)=Exp^+\oplus Exp^-$.\\

\begin{pro}$($\cite[Proposition 6.1]{H}$)$
Let $\pi^{\pm}$ be the projections from $L^2(S^1;\mathbb{C}^2)$ to $Exp^{\pm}$ and $p^{\pm}:=\pi^{\pm}\circ B$ which form the following diagram
\begin{diagram}
      &    &     Exp^+  \\
      &\ruTo^{p^+}    &    \uTo \mbox{ }\pi^+   \\
ker(D|_{L^2(M-\Sigma;\mathcal{S}_{g,\Sigma}})&    \rTo^{\mbox{ } \mbox{ } \mbox{ }B\mbox{ } \mbox{ } \mbox{ }}        &L^2(S^1;\mathbb{C}^2) \\
     &\rdTo_{p^-}     &    \dTo \mbox{ }\pi^-   \\
     &    &    Exp^-,
\end{diagram}
then $p^+$ is a compact operator and $p^-$ is a Fredholm operator.
\end{pro}

\begin{remark}
Whenever $\mathcal{S}_{S^1}|_{N}$ is a trivial or nontrivial complex line bundle, we always have $B$ mapping to $L^2(S^1;\mathcal{S}_{\Sigma}\oplus\mathcal{S}_{\Sigma})$ for $\mathcal{S}_{\Sigma}$ being a spinor bundle on $\Sigma$. The same argument works with $l\in \mathbb{Z}+\frac{1}{2}$ when it is nontrivial. So we only need to focus on one case.
\end{remark}

\subsection{Linearization of $\mathfrak{M}$}Let $p=(g_0,\Sigma_0,\psi_0)\in \mathfrak{M}$. We choose $(g_s,\Sigma_s,\psi_s)$ to be a $C^1$-curve passing through this point in $\mathcal{Y}\times L^2_1(\mathcal{S}_{g,\Sigma})$ with $s\in (-\varepsilon,\varepsilon)$. To be more specific, firstly, one can parametrize the tubular neighbourhood of $\Sigma_0$ by $\{(t, z)|t\in [0,2\pi] \mbox{ and }z\in\mathbb{C}, |z|<R\}$ for some small $R$. We call this neighborhood $N$. Under this coordinate, we write
\begin{align*}
\Sigma_s&=\{(t, s\eta(t)+O(s^2))\},\\
\psi_s&=\psi_0(t, z-s\eta+O(s^2))+s\phi_s
\end{align*}
for some $C^1$-map $\eta:S^1\rightarrow \mathbb{C}$ with $\|\eta\|_{C^1}\leq 1$ and $\phi_s=O_{L^2_1}(1)$. Here we use the notation $O_{L^2_1}(1)$ to denote a one-parameter section $\rho_s$ satisfying $\|\rho_s\|_{L^2_1}\leq C$ for some constant $C>0$. We also choose $\varepsilon$ small enough such that $\Sigma_s\subset N$ for all $s\in (-\varepsilon,\varepsilon)$.\\

In addition, the metric perturbation can be written as
\begin{align*}
g_s=g_0+s\delta_s
\end{align*}
which satisfies $\delta_s=0$ on the tubular neighborhood $N$.\footnote{This is part of assumption we used in [8]. We assumed the metric perturbation and the perturbation of $\Sigma$ will not interfere with each other when $s$ small.} Let $D_s$ be the Dirac operator defined on $M-\Sigma_s$ with respect to $g_s$ and $D=D_0$, then we have $D_s=D$ on $N-(\Sigma_s\cup \Sigma_0)$. So we can write $D_s=D+sT_s$ for some first order differential operator $T_s$ defined on $M-\Sigma_s$ supported on $M-N$.\\

Since $\psi_0$ vanishes on $\Sigma_0$ and satisfies the Dirac equation, we can write down the general solution for it of the form
\begin{align*}
\psi_0=\left(\begin{array}{c}
d^+(t)\sqrt{z}\\
d^-(t)\sqrt{\bar{z}}
\end{array}\right)+O(|z|^\alpha)
\end{align*}
for some $\alpha>\frac{1}{2}$ by Proposition 2.1. So
\begin{align*}
\psi_s=\left(\begin{array}{c}
d^+(t)\sqrt{z-s\eta+O(s^2)}\\
d^-(t)\sqrt{\bar{z}-s\bar{\eta}+O(s^2)}
\end{array}\right)+O(|z-s\eta+O(s^2)|^\alpha)+s\phi_s.
\end{align*}

Now, for any $(\delta_0, \eta, \phi_0)$ defined as above, we have map
\begin{align*}
\mathfrak{L}_p(\delta_0, \eta, \phi_0):&=\frac{d}{ds}(D_s\psi_s)\Big{|}_{s=0}=T_0(\psi_0)+D\Big{(}\frac{d}{d s}\psi_s\Big{)}\Big{|}_{s=0}\\
&=T_0(\psi_0)+D\bigg(\left(\begin{array}{c}
\frac{d^+(t)\eta}{\sqrt{z}}\nonumber\\
\frac{d^-(t)\bar{\eta}}{\sqrt{\bar{z}}}
\end{array}\right)+\mathcal{R}_p(\eta)+\phi_0\bigg).
\end{align*}
Here $\mathcal{R}_p(\eta)$ is an element determined by $p=(g_0,\Sigma_0,\psi_0)$ and $\eta$. We notice that $\mathcal{R}_p(\eta)=O_{L^2_1}(1)$. Also notice that $\delta_0$ corresponds to the metric perturbation ($p_1(\mathcal{N})$ part) in Theorem 1.1. Since we are now interested in the space $\mathfrak{M}_{g_0}$, we can take $\delta_0=0$ here. Namely, $T_0(\psi_0)=0$. So we define the linearization map
\begin{align}
\mathfrak{L}_{p}:\{\eta:S^1\rightarrow \mathbb{C};\|\eta\|_{C^1}\leq 1\}\times L^2_1(M-\Sigma; \mathcal{S}_{g,\Sigma})\rightarrow L^2 (M-\Sigma; \mathcal{S}_{g,\Sigma})\nonumber\\
\mathfrak{L}_{p}(\eta,\phi):= D\bigg(\left(\begin{array}{c}
\frac{d^+(t)\eta}{\sqrt{z}}\\
\frac{d^-(t)\bar{\eta}}{\sqrt{\bar{z}}}
\end{array}\right)+\mathcal{R}_p(\eta)+\phi_0\bigg)
\end{align}
with $\mathcal{R}_p(\eta)=O_{L^2_1}(1)$.\\

Here we study $ker(\mathfrak{L}_{p})$. To satisfy $\mathfrak{L}_{p}(\eta,\phi_0)=0$, we need
\begin{align}
D\bigg(\left(\begin{array}{c}
\frac{d^+(t)\eta}{\sqrt{z}}\\
\frac{d^-(t)\bar{\eta}}{\sqrt{\bar{z}}}
\end{array}\right)+\mathcal{R}_p(\eta)+\phi_0\bigg)=0.
\end{align}

To study the condition (2.2), we use the map $B: ker(D|_{L^2})\rightarrow L^2(S^1;\mathbb{C}^2)$ sending a $L^2$-harmonic spinor to its leading coefficient. In our case, we have
\begin{align*}
B\bigg(\left(\begin{array}{c}
\frac{d^+(t)\eta}{\sqrt{z}}\\
\frac{d^-(t)\bar{\eta}}{\sqrt{\bar{z}}}
\end{array}\right)+\mathcal{R}_p(\eta)+\phi_0\bigg)=(d^+\eta,d^-\bar{\eta}).
\end{align*}
Therefore, to fulfill the equation (2.2), we need
\begin{align}
(d^+\eta,d^-\bar{\eta})\in range(B).
\end{align}

The condition (2.3) still involves the unknown $\eta$, so we define the following map.
\begin{definition}
Let $\psi$ be a $\mathbb{Z}/2$-harmonic spinor. Denoted by $d^{\pm}$ its leading
coefficient as in Proposition 2.1 c). Define
\begin{align*}
&\mathcal{T}_{d^{\pm}}:L^2(S^1;\mathbb{C}^2) \rightarrow L^2(S^1;\mathbb{C})\mbox{ by}\\
&\mathcal{T}_{d^{\pm}}(a,b)=\bar{d}^-a-d^+\bar{b}.
\end{align*}
\end{definition}

Composing $\mathcal{T}_{d^{\pm}}$ with $B$, then we have the following sequence
\begin{align*}
ker(D|_{L^2})\rTo^{\mbox{ }B\mbox{ }} L^2(S^1;\mathbb{C}^2)\rTo^{\mbox{ }\mathcal{T}_{d^{\pm}}\mbox{ }} L^2(S^1;\mathbb{C}).
\end{align*}
Clearly we have $\mathcal{T}_{d^{\pm}}((d^+\eta,d^-\bar{\eta}))=0$. Therefore, we have the following map from $ker(\mathfrak{L}_p)$ to $ker(\mathcal{T}_{d^{\pm}}\circ B)$,
\begin{align*}
 (\eta, \phi_0)\in ker(\mathfrak{L}_p)\longrightarrow u=\left(\begin{array}{c}
\frac{d^+(t)\eta}{\sqrt{z}}\\
\frac{d^-(t)\bar{\eta}}{\sqrt{\bar{z}}}
\end{array}\right)+\mathcal{R}_p(\eta)+\phi_0.
\end{align*}
Here we prove that this map is a bijection by writing down its inverse. For any $u\in ker(\mathcal{T}_{d^{\pm}}\circ B)$, we can write $B(u)=(u^+,u^-)$. So we can solve $\eta=\frac{u^+}{d^+}=\frac{\bar{u}^-}{d^-}$. This is well-defined because $B(u)$ is in $ker(\mathcal{T}_{d^{\pm}})$. With this $\eta$, we can solve $\phi_0$:
\begin{align*}
\phi_0:=u-\left(\begin{array}{c}
\frac{d^+(t)\eta}{\sqrt{z}}\\
\frac{d^-(t)\bar{\eta}}{\sqrt{\bar{z}}}
\end{array}\right)+\mathcal{R}_p(\eta).
\end{align*}
Therefore, there is an inverse map from $ker(\mathcal{T}_{d^{\pm}}\circ B)$ to $ker(\mathfrak{L}_p)$. So we define $\mathbb{K}_0=ker(\mathcal{T}_{d^{\pm}}\circ B)$.\\

In addition, we can prove the following proposition. 
\begin{pro}
$coker(\mathfrak{L}_p)$ is isomorphic to $coker(\mathcal{T}_{d^{\pm}}\circ B)\oplus(ker(D|_{L^2_1}))$.
\end{pro}
The proof of this result is deferred to the appendix. With these correspondences, we have the definition of $\mathbb{K}_0$ and $\mathbb{K}_1$:
\begin{align*}
\mathbb{K}_0&:=ker(\mathcal{T}_{d^{\pm}}\circ B);\\
\mathbb{K}_1&:=coker(\mathcal{T}_{d^{\pm}}\circ B)\times (ker(D|_{L^2_1})).
\end{align*}
Moreover, for any $p=(g_0,\Sigma_0,\psi_0)\in \mathfrak{M}$, we define the Fredholm operator $F_p$ to be
\begin{align*}
F_p:ker(D|_{L^2})&\rightarrow L^2(S^1;\mathbb{C})\oplus ker(D|_{L^2_1});\\
           u      &\mapsto (\mathcal{T}_{d^{\pm}}\circ B(u),0).
\end{align*}
Then $\mathbb{K}_0$ and $\mathbb{K}_1$ are kernel and cokernel of $F_p$ respectively. It is also clear that $index(\mathfrak{L}_p)=index(F_p)$.\\

Therefore we have the following graph,
\begin{diagram}
      &    &     Exp^+  & & \\
      &\ruTo^{p^+}    &    \uTo \mbox{ }\pi^+   &\rdTo^{\mathcal{T}_{d^{\pm}}|_{Exp^+}} &\\
ker(D|_{L^2(M-\Sigma;\mathcal{S}_{g,\Sigma})})&    \rTo^{\mbox{ } \mbox{ } \mbox{ }B\mbox{ } \mbox{ } \mbox{ }}        &L^2(S^1;\mathbb{C}^2) &\rTo^{\mbox{ } \mbox{ } \mbox{ }\mathcal{T}_{d^{\pm}}\mbox{ } \mbox{ } \mbox{ }} & L^2(S^1)\\
     &\rdTo_{p^-}    &   \dTo \mbox{ }\pi^-    &\ruTo_{\mathcal{T}_{d^{\pm}}|_{Exp^-}} &\\
     &    &    Exp^-   &  &
\end{diagram}
\begin{pro}$($\cite[Theorem 6.12]{H}$)$
$\mathcal{T}_{d^{\pm}}|_{Exp^-}$ is a Fredholm operator and\\ $index(\mathcal{T}_{d^{\pm}}|_{Exp^-})=0$.
\end{pro}
Now, by Proposition 2.2, we have $B=p^++p^-$ where $p^+$ is compact and $p^-$ is Fredholm. Therefore $\mathcal{T}_{d^{\pm}}\circ B$ is a Fredholm operator because $\mathcal{T}_{d^{\pm}}\circ B= \mathcal{T}_{d^{\pm}}|_{Exp^-}\circ p^-+\mathcal{T}_{d^{\pm}}|_{Exp^+}\circ p^+$ where the former is a composition of Fredholm operators and the later is a composition with a compact operator. This implies that $F_p$ is Fredholm. So
\begin{align*}
index(F_p)&=index(\mathcal{T}_{d^{\pm}}\circ B)\\
&=index(\mathcal{T}_{d^{\pm}}|_{Exp^-}\circ p^-)=index(\mathcal{T}_{d^{\pm}}|_{Exp^-})+index(p^-).
\end{align*}
By Proposition 2.6, we have
\begin{align*}
index(\mathcal{T}_{d^{\pm}}\circ B)= index(p^-).
\end{align*}
Therefore, to prove Theorem 1.2, we have to show that the following proposition is true. 
\begin{pro}
$index(p^-)=-dim(ker(D|_{L^2_1}))$.
\end{pro}
\begin{remark}
Recall that the index for Fredholm operators will be an invariant on a connected component. Namely, when we compute the index, we can assume that the metric defined on a small tubular neighborhood of $\Sigma$ is Euclidean. So the Dirac operator defined on the tubular neighborhood can be written as
\begin{align}
D=\left( \begin{array}{cc}
-i & 0\\
0 & i 
\end{array} \right)\partial_t+\left( \begin{array}{cc}
0 & 1\\
0 & 0 
\end{array} \right)\partial_z+\left( \begin{array}{cc}
0 & 0\\
-1 & 0 
\end{array} \right)\partial_{\bar{z}}
\end{align}
where $z=x+iy$ (we rewrite the Dirac operator $D = e_0\partial_t+e_1\partial_x+e_2\partial_y$ in terms of $\partial_t$, $\partial_z$ and $\partial_{\bar{z}}$).
\end{remark}

\section{Proof of Proposition 2.7}
\subsection{Integration by parts}
First of all, by part a) of Proposition 2.1, we have $L^2(M-\Sigma;\mathcal{S}_{g,\Sigma})=range(D|_{L^2_1})\oplus ker(D|_{L^2})$. The first step is to extend the map $B$ on a suitable subspace in $L^2(M-\Sigma;\mathcal{S}_{g,\Sigma})$ which contains $ker(D|_{L^2})$. Here let us denote the domain of the Dirac operator on $L^2$ by $Dom(D)$ (for the detail readers can see the p. 91 in [12]). So for any element  $v\in Dom(D)$, we have $Dv\in L^2$.\\

In addition, recall that we parametrize the tubular neighborhood $N-\Sigma\simeq T^2\times (0,R)$ by $(t,z)\in S^1\times \mathbb{D}_R$. If we use the polar coordinate $z=r e^{i\theta}$, for any continuous section $v$ and $r_0\in(0,R)$, $v(r_0,\cdot,\cdot)$ will be a section defined on the bundle $\mathcal{S}_{g,\Sigma}|_{\{r=r_0\}}\rightarrow T^2$. We discussed in the second paragraph of Section 2.1, $\mathcal{S}_{g,\Sigma}|_{\{r=r_0\}}\simeq (\mathcal{S}\otimes \mathcal{I}_{\Sigma})\oplus  (\mathcal{S}\otimes \mathcal{I}_{\Sigma})$. So $v(r_0,\cdot)$ can be regarded as a section on $(\mathcal{S}\otimes \mathcal{I}_{\Sigma})\oplus  (\mathcal{S}\otimes \mathcal{I}_{\Sigma})$. Again, here we can just consider the case that $\mathcal{S}$ is trivial complex line bundle because the general case has the same argument.
\begin{definition}
Let 
\begin{align*}
E_{\partial}=\Big\{v\in Dom(D)\ \Big|\  r^{\frac{1}{2}}&v(r,\cdot)\rightharpoonup X\mbox{ as } r\rightarrow 0,\\
&\mbox{ for some }X\in (L^2(S^1;\mathbb{C})\otimes e^{-\frac{1}{2}i\theta})\oplus (L^2(S^1;\mathbb{C})\otimes e^{\frac{1}{2}i\theta})
\Big\}.
\end{align*}
Here
\begin{align*}
&L^2(S^1;\mathbb{C})\otimes e^{\frac{1}{2}i\theta}:=\{v\otimes e^{\frac{1}{2}i\theta}\in \mathcal{S}\otimes \mathcal{I}_{\Sigma}|v\in L^2(S^1;\mathbb{C})\};\\
&L^2(S^1;\mathbb{C})\otimes e^{-\frac{1}{2}i\theta}:=\{v\otimes e^{-\frac{1}{2}i\theta}\in \mathcal{S}\otimes \mathcal{I}_{\Sigma}|v\in L^2(S^1;\mathbb{C})\}.
\end{align*}
\end{definition}
Here the limit is in weak sense. The existence of this limit is equivalently to say: When we write $v=(v^+,v^-)$ on the tubular neighborhood of $\Sigma$,
\begin{align}
\left(\begin{array}{c}
\sqrt{z}v^+\\
\sqrt{\bar{z}}v^-
\end{array}\right) \rightharpoonup \left(\begin{array}{c}
y^+\\
y^-
\end{array}\right)\in L^2(S^1;\mathbb{C}^2).
\end{align}
for some $y^{\pm}$ as $r=|z|$ goes to 0.\\

We denote by
\begin{align*}
\partial(v)
\end{align*}
$X$ when the weak limit exist. Meanwhile, we can extend the map $B$ on $E_{\partial}$ by using (3.1). When $v\in ker(D|_{L^2})$, this limit exists and equals $B(v)$. 
We can see that $B(E_{\partial})=L^2(S^1;\mathbb{C}^2)$ because for any $Y=(y^+,y^-)\in L^2(S^1;\mathbb{C}^2)$, then there exists $\partial(u)=Y$ with
\begin{align*}
u=\left(\begin{array}{c}
\frac{y^+}{\sqrt{z}}\\
\frac{y^-}{\sqrt{\bar{z}}}
\end{array}\right)\chi\in E_{\partial}
\end{align*}
where $\chi$ is a continuous function with value 1 near $\Sigma$ and 0 on $M-N$. Accordingly, this new domain we chose maps onto the space $L^2(S^1;\mathbb{C}^2)$. By using this fact and part a) of Proposition 2.1, for any $Y\in B(ker(D|_{L^2}))^{\perp}$, there exists an element $w\in range(D|_{L^2_1})$ such that $Y-B(w)\in B(ker(D|_{L^2}))$.\\

Secondly, we consider the integration by parts. Let $v,w\in E_{\partial}$. Then we have
\begin{align}
\int_{M-\Sigma} \langle Dv,w\rangle+\langle v, Dw\rangle=\int_0^{2\pi}\int_{S^1}\langle \partial(v), e_*\partial(w)\rangle dt d\theta
\end{align}
where $e_*$ is the Clifford multiplication $cl(\partial_r)=\left( \begin{array}{cc}
0 & e^{-i\theta}\\
-e^{i\theta} & 0 
\end{array} \right)$ by using the notation in (2.4). So (3.2) can be written as
\begin{align}
\int_{M-\Sigma} \langle Dv,w\rangle+\langle v, Dw\rangle=2\pi\int_{S^1}\langle B(v), e_0B(w)\rangle dt
\end{align}
where $e_0=\left( \begin{array}{cc}
0 & 1\\
-1 & 0 
\end{array} \right)$.

 Notice that the Clifford multiplication $e_0$ can be regarded as a map from $Exp^{\pm}$ to $Exp^{\mp}$. So we can define the following nondegenerate bilinear form
\begin{align*}
\mathscr{B}(X,Y)=\int_{S^1}\langle X,e_0Y\rangle
\end{align*}
on $L^2(S^1)$. Meanwhile, we also have the standard inner product
\begin{align*}
(X,Y)=\int_{S^1}\langle X,Y\rangle
\end{align*}
on $L^2(S^1)$, and we write $X\perp Y$ if and only if $(X,Y)=0$.\\

\begin{pro}
$L^2(S^1;\mathbb{C}^2)=B(ker(D|_{L^2}))\oplus e_0B(ker(D|_{L^2}))$. Namely, $B(ker(D|_{L^2}))$ can be regarded as a Lagrangian subspace of $L^2(S^1;\mathbb{C}^2)$.
\end{pro}

\begin{proof}
To prove this proposition, by using equality (3.3), we have 
\begin{align*}
B(ker(D|_{L^2}))\perp e_0B(ker(D|_{L^2})).
\end{align*}
This implies that $B(ker(D|_{L^2}))^{\perp}\subseteq e_0B(ker(D|_{L^2}))$. So one can prove this proposition by showing that $B(ker(D|_{L^2}))^{\perp}= e_0B(ker(D|_{L^2}))$. In addition, every element in $B(ker(D|_{L^2}))^{\perp}$ can be written as $B(Du)+B(v)$ for some $u\in L^2_1$ and $v\in ker(D|_{L^2})$. Therefore, to prove $B(ker(D|_{L^2}))^{\perp}= e_0B(ker(D|_{L^2}))$, one needs to show that if there is a $B(Du)+B(v)\in B(ker(D|_{L^2}))^{\perp}$ such that $\mathscr{B}(B(Du)+B(v), Y)=0$ for all $Y\in B(ker(D|_{L^2}))$, then $B(Du)+B(v)=0$.\\

 Since $v\in ker(D|_{L^2})$, we always have $\mathscr{B}(B(v),Y)=0$. So we can rewrite our assumption as follows
\begin{align*}
\mathscr{B}(B(Du),B(w))=0
\end{align*}
for all $w\in ker(D|_{L^2})$. By (3.3) again,
\begin{align*}
\int_{M-\Sigma}\langle D^2u, w\rangle=0
\end{align*}
for all $w\in ker(D|_{L^2})$. So we have $D^2u\in ker(D|_{L^2})^{\perp}=range(D|_{L^2_1})$. This means that
 $D^2u=Du'$ for some $u'\in L^2_1$. Therefore $Du=u'+v'$ for some $v'\in ker(D|_{L^2})$, which implies that $B(Du)=B(v')\in B(ker(D|_{L^2}))$. Eventually, we have $B(Du)+B(v)\in B(ker(D|_{L^2}))\cap  B(ker(D|_{L^2}))^{\perp}=\{0\}$. So we prove this proposition.
\end{proof}

 Now the following fact can be derived immediately from this proposition: 
\begin{align*}
coker(p^-)&=[p^-(ker(D|_{L^2}))]^{\perp}=[\pi^-\circ B(ker(D|_{L^2}))]^{\perp}\\
&=B(ker(D|_{L^2}))^{\perp}\cap ker(\pi^+)=e_0B(ker(D|_{L^2}))\cap ker(\pi^+)\\
&=\{v\in Exp^-|v\in e_0B(ker(D|_{L^2}))\}\\
&=\{e_0v\in Exp^+|e_0v\in B(ker(D|_{L^2}))\}\\
&=B(ker(D|_{L^2})\cap Exp^+
\end{align*}
(also by the fact that $e_0^2=-1$).

Here we prove $ker(p^-)/ker(D|_{L^2_1}) \cong B(ker(D|_{L^2})\cap Exp^+$. If we take the quotient of $p^+:ker(p^-)\rightarrow B(ker(D|_{L^2})\cap Exp^+$  by its kernel $ker(D|_{L^2_1})$, we have an injective map from $ker(p^-)/ker(D|_{L^2_1})$ to $B(ker(D|_{L^2})\cap Exp^+$. It is obvious by the definition of its range that the this map is onto.  This means $B(ker(D|_{L^2})\cap Exp^+\cong ker(p^-)/ker(D|_{L^2_1})$. This completes the proof of Proposition 2.8.

\section{4-dimensional setting}
\subsection{Main setting}
In this section we consider the 4-dimensional generalization of the index theorem with respect to the the $\mathbb{Z}/2$-harmonic spinors. Let $M$ be a closed oriented smooth 4-manifold with the second Stiefel-Whitney class $w_2=0$. $\mathcal{X}$ be the space of Riemannian metrics defined on $M$. In this case, for any $g\in \mathcal{X}$, there exists a (not necessarily unique) spinor bundle $\mathcal{S}=\mathcal{S}^+\oplus\mathcal{S}^-$. 
\begin{align*}
\mathcal{A}_{T^2}=\{ C^1\mbox{-embedding surface } &\Sigma\subset M \mbox{ with trivial normal bundle},\\
&\mbox{ }\mbox{ }\mbox{ }\mbox{ }\mbox{ }\mbox{ }\mbox{ }\mbox{ }\mbox{ }\mbox{ }\mbox{ }\mbox{ }\mbox{ }\Sigma\mbox{ is homeomorhic to }T^2\}.
\end{align*}
Let $\Sigma\in \mathcal{A}_{T^2}$ and $g\in \mathcal{X}$. Recall that a $\mathbb{Z}/2$-spinor bundle with respect to $(g, \Sigma)$ is a spinor bundle which can be written as $\mathcal{S}_g\otimes \mathcal{I}_{\Sigma}$, where $\mathcal{S}_g$ is a spinor bundle over $(M,g)$ and $\mathcal{I}_{\Sigma}$ is a non-extendable real line bundle over $M-\Sigma$. Again, we use $\mathcal{S}_{g,\Sigma}$ to denote one of these bundles. Moreover, because there is a standard decomposition $\mathcal{S}_g= \mathcal{S}^+_g\oplus \mathcal{S}_g^-$, we have $\mathcal{S}_{g,\Sigma}=\mathcal{S}_{g,\Sigma}^+\oplus\mathcal{S}_{g,\Sigma}^-$ accordingly.\\

The Dirac operator $D$ on $\mathcal{S}_{g,\Sigma}$ can also be decomposed as $D=D^+\oplus D^{-}$ where $D^{\pm}$ map $\mathcal{S}^{\pm}_{g,\Sigma}$ to $\mathcal{S}^{\mp}_{g,\Sigma}$. We consider one of them, say $D^+$, and define the moduli space as the following:
\begin{align*}
\mathfrak{M}_{T^2}=\{(\psi,\Sigma,g)|&D^+(\psi)=0, \psi\in C^{\infty}(\mathcal{S}^+_{g,\Sigma})\\
& |\psi| \mbox{ can be extended as a H\"older continuous function on } M ,\\
 &\mbox{ }\mbox{ }\mbox{ with its zero locus
containing } \Sigma,\\
& \frac{|\psi|(p)}{\mbox{dist}(p,\Sigma)^{\frac{1}{2}}}> 0\mbox{ near }\Sigma,\\
& \|\psi\|_{L^2_1}>0.\}
\end{align*}
and $\mathfrak{M}_{T^2,g_0}=\mathfrak{M}_{T^2}\cap \{g=g_0\}$.\\

In general, we can define the moduli space $\mathfrak{M}_{X}$ for any Riemann surface $X$. In fact, we will have the same index theorem as the case $X=T^2$. However in this paper we focus on this special case because we can precisely write down the model solution for Dirac equation in the tubular neighborhood of $\Sigma$.\\

\subsection{Linearization of $\mathfrak{M}_{T^2}$}
To prove a four dimensional version of Theorem 1.2, we should start with the linearization of $\mathfrak{M}_{T^2}$. This part has the same structure as the 3-dimensional case. Consider the model of the tubular neighborhood, $T^2\times  \mathbb{D}_R$ where $\mathbb{D}_R$ is a complex disc of radius $R>0$ in $\mathbb{C}$, the Dirac operator can be written as
\begin{align}
D^+=e_0 \hat{D}+ e_1 \partial_z+ e_2 \partial_{\bar{z}}
\end{align}
where $\hat{D}$ is the Dirac operator defined on $T^2$ and $e_0, e_1, e_2$ are Clifford multiplications  with $e_0=\left( \begin{array}{cc}
0 & 1\\
-1 & 0 
\end{array} \right)$, $e_1=\left( \begin{array}{cc}
0 & 1\\
0 & 0 
\end{array} \right)$ and $e_2=\left( \begin{array}{cc}
0 & 0\\
-1 & 0 
\end{array} \right)$.\\

Note that, in the 3-dimensional case, a general solution of Dirac equation (2.4) can be written as follows: For any $C^{\infty}$-spinor $\mathfrak{u}$, it can be written as a Fourier series 
\begin{align*}
\mathfrak{u}(t,r,\theta)=\sum_{l,k}e^{ilt}\left( \begin{array}{c}
e^{i(k-\frac{1}{2})\theta}U^+_{k,l}(r)\\
e^{i(k+\frac{1}{2})\theta}U^-_{k,l}(r)
\end{array} \right)
\end{align*}
where $k$ runs over $\mathbb{Z}$ and $l$ runs over $\mathbb{Z}$ or $\mathbb{Z}+\frac{1}{2}$. Then $U^{\pm}_{k,l}$ will satisfy an ODE provided by the Dirac equation $D\mathfrak{u}=0$. It can be written as
\begin{align}
\frac{d}{dr}\left( \begin{array}{c}
U^{+}\\
U^{-}
\end{array} \right)_{k,l}
=
\left( \begin{array}{cc}
\frac{(k-\frac{1}{2})}{r}& -l\\
-l & -\frac{(k+\frac{1}{2})}{r}
\end{array} \right)
\left( \begin{array}{c}
U^{+}\\
U^-
\end{array} \right)_{k,l}.
\end{align}

The situation is similar in 4 dimensional case: Parametrizing $T^2\times \mathbb{R}^2$ by $\{(x,r,\theta)\in T^2\times \mathbb{R}_{\geq 0}\times [0,2\pi]\}$, we can write a $C^{\infty}$-section $\mathfrak{u}$ as follows
\begin{align*}
\mathfrak{u}(x,r,\theta)=\sum_{a,k}\left( \begin{array}{c}
v^\wedge_a(x)e^{i(k-\frac{1}{2})\theta}U^+_{k,a}(r)\\
v^\vee_a(x)e^{i(k+\frac{1}{2})\theta}U^-_{k,a}(r)
\end{array} \right).
\end{align*}
Here $k$ still runs over $\mathbb{Z}$ and $a$ runs over $\Lambda$, the eigenvalues of $\hat{D}$ (counting repeatedly if we have repeat eigenvalues). $v_a=(v^\wedge_a,v^\vee_a)$ satisfies $e_0\hat{D}^+v^\wedge_a=av^\vee_a$ and $e_0\hat{D}^-v^\vee_a=av^\wedge_a$. $\{v^\wedge_a\}$ $\{v^\vee_a\}$ will be  orthonormal bases of $L^2(\mathcal{S}^+_{\Sigma})$ and $L^2(\mathcal{S}^-_{\Sigma})$ respectively.\\

 In our case that $\Sigma\simeq S^1\times S^1$ equipped with Euclidean metric, we can write down these $v_a=(v^\wedge_a,v^\vee_a)$ precisely. Since we have assumed the Fredholmness of the linearization, the index wouldn't change under any perturbation of metrics. Therefore one can obtain the index formula under this assumption. Let us consider the Dirac operator $D^+$ with respect to the standard flat metric $dt^2+ds^2+dr^2+rdrd\theta+d\theta^2$ and $\mathcal{S}^{\pm}_{\Sigma}$ are trivial, then we have
\begin{align*}
D^+=\left(\begin{array}{cc}
1& 0\\
0& 1
\end{array}\right)\partial_t+
\left(\begin{array}{cc}
-i& 0\\
0& i
\end{array}\right)\partial_s+\left(\begin{array}{cc}
0& 1\\
0& 0
\end{array}\right)\partial_z+\left(\begin{array}{cc}
0& 0\\
-1& 0
\end{array}\right)\partial_{\bar{z}}.
\end{align*}
The sum of the first two terms is $e_0\hat{D}$ defined above. So we can define
\begin{align*}
\Big\{v_{l,m}:=(e^{ilt} e^{ims},\frac{-il+m}{\sqrt{l^2+m^2}} e^{ilt} e^{ims})\Big|(l,m)\in\mathbb{Z}\times\mathbb{Z}-(0,0)\Big\}
\end{align*}
and $\Lambda:=\{\sqrt{l^2+m^2}|(l,m)\in\mathbb{Z}\times\mathbb{Z}-(0,0)\}$. When $\mathcal{S}^{\pm}_{\Sigma}$ are non-trivial, we can simply replace those $\mathbb{Z}$ by $\mathbb{Z}+\frac{1}{2}$ respectively according to the non-triviality of $\mathcal{S}^{\pm}_{\Sigma}$.\\

In the following paragraphs, we define
\begin{align*}
sign(l,m):=\frac{l+im}{\sqrt{l^2+m^2}}.
\end{align*}
So
\begin{align*}
v_{l,m}=(e^{ilt}e^{ims},-sign(l,m)ie^{ilt}e^{ims}).
\end{align*}
Notice that this sign function $sign(l,m)$ can be regraded as a generalized sign for paring numbers: We have $sign(l,0)=sign(l)$ and $sign(0,m)=isign(m)$.\\

 Given any $k,a=\sqrt{l^2+m^2}$, $U^{\pm}_{k,a}$ will satisfy the same ODE system,
\begin{align}
\frac{d}{dr}\left( \begin{array}{c}
U^{+}\\
U^{-}
\end{array} \right)_{k,a}
=
\left( \begin{array}{cc}
\frac{(k-\frac{1}{2})}{r}& -a\\
-a & -\frac{(k+\frac{1}{2})}{r}
\end{array} \right)
\left( \begin{array}{c}
U^{+}\\
U^-
\end{array} \right)_{k,a},
\end{align}
as they did in the 3-dimensional case. By solving this ODE system, we will have
\begin{align}
\mathfrak{u}(x,r,\theta)=\sum_{k,a}&
\left( \begin{array}{c}
u^{+}_{k,a}v^\wedge_a(x)e^{i(k-\frac{1}{2})\theta}\mathfrak{I}_{k-\frac{1}{2},a}(r)\\
-u^{+}_{k,a}v^\vee_a(x)e^{i(k+\frac{1}{2})\theta}a\mathfrak{I}_{k+\frac{1}{2},a}(r)
\end{array} \right)\\
+&
\left( \begin{array}{c}
-u^{-}_{k,a}v^\wedge_a(x)e^{i(k-\frac{1}{2})\theta}a\mathfrak{I}_{k-\frac{1}{2},a}(r)\\
u_{k,a}^{-}v^\vee_a(x)e^{i(k+\frac{1}{2})\theta}\mathfrak{I}_{k+\frac{1}{2},a}(r)
\end{array} \right)\nonumber
\end{align}
for some $u^{\pm}_{k,a}\in \mathbb{C}$. Here $\mathfrak{I}_{p,a}(r):=a^{-p}\sum_{n=0}^{\infty}\frac{1}{n!\Gamma(n+p+1)}(\frac{ar}{2})^{2n+p}$ is the modified Bessel function (when $a=0$, we simply take $\mathfrak{I}_{p,0}(r):=r^p$).\\

Now, recall that these modified Bessel functions have order $\mathfrak{I}_{p,a}(r)=O(r^p)$. So if $\mathfrak{u}(x,r,\theta)\in L^2$, then $u^{\pm}_{k,a}=0$ for all $k<0$ and the leading order term of $\mathfrak{u}$ will be of order $O(r^{-\frac{1}{2}})$. Similarly, if $\mathfrak{u}\in L^2_1$, then the leading order term of $\mathfrak{u}$ will be of order $O(r^{\frac{1}{2}})$. Therefore, $b)$ and $c)$ in Proposition 2.1 can be derived in 4 dimensional case. In other words, we have
\begin{pro}
\ \\
$\mbox{a).}$ $ L^2(M-\Sigma;\mathcal{S}^+_{g,\Sigma})=ker(D^+|_{L^2})\oplus range(D^-|_{L^2_1})$,\\
$\mbox{b).}$ For any $v\in ker(D^+|_{L^2})$, we have
\begin{align*}
v=\left(\begin{array}{c}
\frac{c^+(x)}{\sqrt{z}}\\
\frac{c^-(x)}{\sqrt{\bar{z}}}
\end{array}\right)+v_{R},
\end{align*}
$\mbox{ }\mbox{ }\mbox{ }$    where $v_{R}=O(|z|^{\alpha})$ for some $\alpha >0 $ and $c^{\pm}\in C^{\infty}(N-\Sigma, \mathcal{S}^{\pm}_{T^2}\otimes\mathcal{I}_{\Sigma})$.\\
$\mbox{c).}$ For any $u\in ker(D^+|_{L^2_1})$, we have
\begin{align*}
u=\left(\begin{array}{c}
 d^+(x)\sqrt{z}\\
 d^-(x)\sqrt{\bar{z}}
\end{array}\right)+u_R.
\end{align*}
$\mbox{ }\mbox{ }\mbox{ }$    where $u_{R}=O(|z|^{\beta})$ for some $\beta >\frac{1}{2}$ and $d^{\pm}\in C^{\infty}(N-\Sigma, \mathcal{S}^{\pm}_{T^2}\otimes\mathcal{I}_{\Sigma})$.\\
Here $\mathcal{S}^{\pm}_{T^2}$ are the pull-back bundles of $\mathcal{S}^{\pm}_{\Sigma}$ by the map $\pi:N-\Sigma\rightarrow \Sigma$ and $\mathcal{S}_{\Sigma}=\mathcal{S}_{\Sigma}^+\oplus \mathcal{S}_{\Sigma}^-$ is a spinor bundle defined on $\Sigma$.
\end{pro}
The proof of this proposition is same as the proof of Proposition 2.1 which can be found in [8]. So we omit it here.\\

Here these leading coefficients $(c^+(x),c^-(x)), (d^+(x),d^-(x))$ are in $L^2(T^2;\mathcal{S}_{\Sigma})\cong L^2(T^2;\mathcal{S}^+_{\Sigma}\oplus \mathcal{S}^-_{\Sigma})$. By Proposition 4.1, for any element $(g,\Sigma,\psi)\in\mathfrak{M}_{T^2}$, the linearization argument in Section 2.2 can be derived. So we have the following composition of maps.
\begin{align}
ker(D^+|_{L^2(M-\Sigma;\mathcal{S^+}\otimes \mathcal{I})})   \rTo^{\mbox{ } \mbox{ } \mbox{ }B\mbox{ } \mbox{ } \mbox{ }}        L^2(T^2;\mathcal{S}^+_{\Sigma}\oplus \mathcal{S}^-_{\Sigma}) \rTo^{\mbox{ } \mbox{ } \mbox{ }\mathcal{T}_{d^{\pm}}\mbox{ } \mbox{ } \mbox{ }}  L^2(T^2;\mathcal{S}^+_{\Sigma}\otimes \overline{\mathcal{S}^-_{\Sigma}}).
\end{align}
The map $B$ in this short sequence is also defined in [9] which will give us a useful index formula in Theorem 4.3. Here we need to explain the map $\mathcal{T}_{d^{\pm}}$ more. As we follow the argument in Section 2.2, we will have
\begin{align*}
d^+\eta=c^+;\\
d^-\bar{\eta}=c^-
\end{align*}
where $\eta$ is a complex value function and $(d^+,d^-), (c^+,c^-)$ are in $L^2(T^2;\mathcal{S}^+_{\Sigma})\oplus L^2(T^2;\mathcal{S}^-_{\Sigma})$. To kill the term on the left hand side of this equation, we tenser both sides of the first equation on the right with the conjugate of $d^-$ in the conjugate bundle of $\mathcal{S}_{\Sigma}^-$, denoted by $\bar{d}^-$. Meanwhile, tensor the conjugate of the second equation on the left with $d^+$. So we have
\begin{align*}
\eta(d^+\otimes \bar{d}^-)=(c^+\otimes \bar{d}^-);\\
\eta(d^+\otimes \bar{d}^-)=(d^+\otimes \bar{c}^-).
\end{align*}
Therefore, we define $\mathcal{T}_{d^{\pm}}$ by
\begin{align*}
\mathcal{T}_{d^{\pm}}(c^+,c^-)=(c^+\otimes \bar{d}^-)-(d^+\otimes \bar{c}^-).
\end{align*}
By the same argument we used in the 3-dimensional case, the linearization of $\mathfrak{M}_{T^2,g}$ can be locally written as a map between the following two spaces:
\begin{align*}
\mathbb{K}_0&=ker(\mathcal{T}_{d^{\pm}}\circ B);\\
\mathbb{K}_1&=coker(\mathcal{T}_{d^{\pm}}\circ B)\times (ker(D^-|_{L^2_1})),
\end{align*}
which are the kernel and cokernel of the map $F_{T^2,p}$,
\begin{align}
F_{T^2,p}:ker(D^+|_{L^2(M-\Sigma;\mathcal{S}^+_{g,\Sigma})})&\rightarrow L^2(T^2;\mathcal{S}^+_{\Sigma}\otimes \overline{\mathcal{S}^-_{\Sigma}})\oplus ker(D|_{L^2_1});\\
           u      &\mapsto ( \mathcal{T}_{d^{\pm}}\circ B(u),0).\nonumber
\end{align}

To mimic the argument in the 3-dimensional case, we shall define the decomposition $\pi^{\pm}$, which appears in the following subsection. 

\subsection{Decomposition of $\pi^{\pm}$} Unlike the 3-dimensional case, here we wouldn't use a symmetric decomposition to make $L^2(T^2;\mathcal{S}^+_{\Sigma}\oplus \mathcal{S}^-_{\Sigma})=Exp^+\oplus Exp^-$. Instead, we follow the idea in [11], developed by Atiyah, Patodi and Singer, to decompose $L^2(T^2;\mathcal{S}^+_{\Sigma}\oplus \mathcal{S}^-_{\Sigma})$ asymmetrically into the following three parts:
\begin{align}
Exp^+=\bigg\{\left(\begin{array}{c}
\sum_{(l,m)\in\mathbb{Z}^2-0} p_{l,m} e^{ilt}e^{ims}\\
\sum_{(l,m)\in\mathbb{Z}^2-0} -sign(l,m)ip_{l,m} e^{ilt}e^{ims}
\end{array}\right)\bigg| (p_{l,m})\in l^2\bigg\},\\
Exp^-=\bigg\{\left(\begin{array}{c}
\sum_{(l,m)\in\mathbb{Z}^2-0} p_{l,m} e^{ilt}e^{ims}\\
\sum_{(l,m)\in\mathbb{Z}^2-0} sign(l,m)ip_{l,m} e^{ilt}e^{ims}
\end{array}\right)\bigg| (p_{l,m})\in l^2\bigg\}
\end{align}
and $ker(D_{\Sigma})$ where $D_{\Sigma}:=\hat{D}$ is the Dirac operator defined on $T^2$. We also denote $Exp^{+}\oplus ker(D_{\Sigma})$ by $Exp^{+,0}$ and  $Exp^{-}\oplus ker(D_{\Sigma})$ by $Exp^{-,0}$.\\

 By using this decomposition, we obtain the following diagram:
\begin{diagram}
      &    &     Exp^{+,0}  &\\
      &\ruTo^{p^{+,0}}    &    \uTo \mbox{ }\pi^{+,0}   \\
ker(D^+|_{L^2(M-\Sigma;\mathcal{S}^+_{g,\Sigma})})&    \rTo^{\mbox{ } \mbox{ } \mbox{ }B\mbox{ } \mbox{ } \mbox{ }}        &L^2(T^2;\mathcal{S}^+_{\Sigma}\oplus \mathcal{S}^-_{\Sigma}) &\\
     &\rdTo_{p^-}    &   \dTo \mbox{ }\pi^-    &\\
     &    &    Exp^-   & 
\end{diagram}

 One can also define $Exp^-$ in the following alternative way. We can consider $D^+$ defined in (4.1) on an extended domain $T^2\times\mathbb{C}$ with respect to the product metric (which uses the standard Euclidean metric on the second component $\mathbb{C}\cong \mathbb{R}^2$). Denote by $L^2(T^2\times \mathbb{C})$ the space of $L^2$ sections on the corresponding expended spinor bundle. Then $Exp^-$ can be written as
\begin{align*}
Exp^-=\{B(u)|u\in ker(D^+|_{L^2(T^2\times \mathbb{C})}), |u|(x,r,\theta)<ce^{-\delta r} \mbox{ for some }c,\delta>0\}.
\end{align*}

Now, instead of using $D^+$, we also have the following diagram for $ker(D^-|_{L^2})$. Again, $L^2(T^2;\mathcal{S}^+_{\Sigma}\oplus \mathcal{S}^-_{\Sigma})$ can be decomposed into $\mathcal{E}xp^{\pm}$ and $ker(D_{\Sigma})$ in the following way:
\begin{align*}
\mathcal{E}xp^-=\{B(u)|u\in ker(D^-|_{L^2(T^2\times \mathbb{C})}), |u|(x,r,\theta)<ce^{-\delta r} \mbox{ for some }c,\delta>0\},
\end{align*}
$\mathcal{E}xp^+:=(\mathcal{E}xp^+\oplus ker(D_{\Sigma}))^{\perp}$ and $\mathcal{E}xp^{\pm,0}:=\mathcal{E}xp^{\pm}\oplus ker(D_{\Sigma})$.\\

Therefore we have
\begin{diagram}
      &    &     \mathcal{E}xp^{+,0}  &\\
      &\ruTo^{\mathfrak{p}^{+,0}}    &    \uTo \mbox{ }\pi^{+,0}   &\\
ker(D^-|_{L^2(M-\Sigma;\mathcal{S}^-_{g,\Sigma})})&    \rTo^{\mbox{ } \mbox{ } \mbox{ }B\mbox{ } \mbox{ } \mbox{ }}        &L^2(T^2;\mathcal{S}^+_{\Sigma}\oplus \mathcal{S}^-_{\Sigma}) &\\
     &\rdTo_{\mathfrak{p}^-}    &   \dTo \mbox{ }\pi^-    &\\
     &    &    \mathcal{E}xp^-   & 
\end{diagram}
\begin{pro}
The operators $p^-$, $\mathfrak{p}^-$ are Fredholm. The operators $p^{+,0}$, $\mathfrak{p}^{+,0}$ are compact. Moreover, $\mathfrak{p}^{-,0}$, the projection from $ker(D^-|_{L^2(M-\Sigma;\mathcal{S}^-_{g,\Sigma})})$ to $\mathcal{E}xp^{-,0}$, is also a Fredholm operator. 
\end{pro}

\begin{proof}
Here we just prove that $p^-$ is Fredholm and $p^{+,0}$ is compact because other cases can be obtained by the same argument.\\

Firstly, we prove $p^-$ is Fredholm. This is equivalently to say that $p^-$ has finite dimensional kernel and finite dimensional cokernel. By the computation in Section 4.2, for any $\mathfrak{u}\in ker(D^+|_{L^2})$, we have
\begin{align*}
\mathfrak{u}=\sum_{l,m}e^{ilt}e^{ims}&\left( \begin{array}{c}
\hat{u}_{l,m}^{+}\frac{e^{\sqrt{l^2+m^2}r}}{\sqrt{z}}+\hat{u}_{l,m}^{-}\frac{e^{-\sqrt{l^2+m^2}r}}{\sqrt{z}}\\
-\mbox{sign}(l,m)i\hat{u}^{+}_{l,m}\frac{e^{\sqrt{l^2+m^2}r}}{\sqrt{\bar{z}}}+\mbox{sign}(l,m)i\hat{u}^{-}_{l,m}\frac{e^{-\sqrt{l^2+m^2}r}}{\sqrt{\bar{z}}}
\end{array} \right)\\
&\mbox{ }\mbox{ }\mbox{ }\mbox{ }\mbox{ }\mbox{ }\mbox{ }\mbox{ }\mbox{ }\mbox{ }\mbox{ }\mbox{ }\mbox{ }\mbox{ }\mbox{ }\mbox{ }\mbox{ }\mbox{ }\mbox{ }\mbox{ }\mbox{ }\mbox{ }\mbox{ }\mbox{ }\mbox{ }\mbox{ }\mbox{ }\mbox{ }\mbox{ }\mbox{ }\mbox{ }\mbox{ }\mbox{ }\mbox{ }\mbox{ }\mbox{ }\mbox{ }\mbox{ }\mbox{ }\mbox{ }\mbox{ }\mbox{ }\mbox{ }\mbox{ }\mbox{ }\mbox{ }\mbox{ }\mbox{ }\mbox{ }\mbox{ }\mbox{ }\mbox{ }\mbox{ }\mbox{ }\mbox{ }\mbox{ }+\mbox{ higher order terms}.
\end{align*}
Let $N_r$ be the tubular neighborhood of $\Sigma$ with thickness $r$. By using Lichnerowicz-Weizenb\"ock formula,
\begin{align*}
\int_{M-N_r}|D^+\mathfrak{u}|^2=\int_{M-N_r}|\nabla \mathfrak{u}|^2+\int_{M-N_r}\langle \mathcal{R}\mathfrak{u},\mathfrak{u} \rangle+\int_{\partial N_r}\langle\mathfrak{u},\partial_r\mathfrak{u}\rangle i_{\partial_r}dVol,
\end{align*}
and taking $r\rightarrow 0$, we have
\begin{align*}
\|\mathfrak{u}\|^2_{L^2_1}\leq \sum_{l,m} \sqrt{l^2+m^2}|\hat{u}_{l,m}^{-}|^2+C\|\mathfrak{u}\|^2_{L^2}.
\end{align*}
Therefore, if $\mathfrak{u}\in ker(p^-)$, then we have $\hat{u}_{l,m}^{-}=0$ for all $l,m$. So
\begin{align*}
\|\mathfrak{u}\|^2_{L^2_1}\leq C\|\mathfrak{u}\|^2_{L^2},
\end{align*}
which implies that kernel $p^-$ is finite dimensional.\\

To prove the cokernel is finite-dimensional, we claim that there exists $N>0$ such that $range(p^-)+\mathbb{V}_N=Exp^-$, where 
\begin{align*}
\mathbb{V}_N=\Big\{\sum_{l,m}(\hat{u}^-_{l,m},-sign(l,m)i\hat{u}^-_{l,m})e^{ilt}e^{ims}\Big|\hat{u}^-_{l,m}=0\mbox{ for all } l^2+m^2>N\Big\}.
\end{align*}
We can easily see that if this claim is true, then the $coker(p^-)$ will be finite dimensional.\\

To prove this claim, we need to prove the following statement first: There exists $N>0$ with the following significance. For any 
\begin{align*}
V=\sum_{l,m}(\hat{u}^-_{l,m},-sign(l,m)i\hat{u}^-_{l,m})e^{ilt}e^{ims}
\end{align*}
with $\hat{u}^-_{l,m}=0$ for all $l^2+m^2<N$, there exists $\mathfrak{u}\in ker(D^+)$ satisfying $\|B(\mathfrak{u})-V\|\leq \frac{1}{3}\|V\|^2$.\\

Here we prove this statement by using a proposition in [8]. We choose
\begin{align*}
\mathfrak{u}_0=\chi(r) \sum_{l^2+m^2>N}e^{ilt}e^{ims}&\left( \begin{array}{c}
\hat{u}_{l,m}^{-}\frac{e^{-\sqrt{l^2+m^2}r}}{\sqrt{z}}\\
\mbox{sign}(l,m)i\hat{u}^{-}_{l,m}\frac{e^{-\sqrt{l^2+m^2}r}}{\sqrt{\bar{z}}}
\end{array} \right),\\
D^+(\mathfrak{u}_0):=\mathfrak{f},\\
B(\mathfrak{u}_0):=V.
\end{align*}
Here $\chi$ is a nonnegative, decreasing function with $\chi(0)=1$, $\chi=0$ on $M-N_R$ for some small $R$. Clearly we have $\|\mathfrak{f}\|_{L^2}^2 \leq C_Re^{-NR}\|V\|^2$. By Proposition 4.3 in [8], there exists $\mathfrak{v}$ such that $D^+(\mathfrak{v})=\mathfrak{f}$ and $\|B(\mathfrak{v})\|^2\leq  C_Re^{-NR}\|V\|^2$. So by taking $\mathfrak{u}=\mathfrak{u}_0-\mathfrak{v}$ and $N$ sufficiently large, we have $D^+\mathfrak{u}=0$ and $\|B(\mathfrak{u})-V\|\leq \frac{1}{3}\|V\|^2$.\\

Now, we prove the claim by using the statement we just proved. Suppose the claim is false, then there exists a non-zero $Y \perp range(p^-)+\mathbb{V}_N$. Suppose $\|Y\|=1$ without loss of generality. Then for any $Z=\sum_{l,m}(\hat{u}^-_{l,m},-sign(l,m)i\hat{u}^-_{l,m})e^{ilt}e^{ims}$, we have
\begin{align*}
|\langle Y, Z\rangle| = &|\langle Y, \sum_{l^2+m^2\leq N}(\hat{u}^-_{l,m},-sign(l,m)i\hat{u}^-_{l,m})e^{ilt}e^{ims}\rangle \\
&+ \langle Y, \sum_{l^2+m^2> N}(\hat{u}^-_{l,m},-sign(l,m)i\hat{u}^-_{l,m})e^{ilt}e^{ims}\rangle| \\
= &|\langle Y, \sum_{l^2+m^2> N}(\hat{u}^-_{l,m},-sign(l,m)i\hat{u}^-_{l,m})e^{ilt}e^{ims}\rangle|\\
\leq & \frac{1}{3}\|\sum_{l^2+m^2> N}(\hat{u}^-_{l,m},-sign(l,m)i\hat{u}^-_{l,m})e^{ilt}e^{ims}\|\\
\leq & \frac{1}{3}\|Z\|.
\end{align*}
However, the $\sup_{\|Y\|=1}|\langle Y, Z\rangle|=\|Z\|$, this leads a contradiction. So $coker(p^-)$ is finite dimensional.\\

Secondly, we have to prove $p^{+,0}$ is compact. Since $ker(D_{\Sigma})$ is finite dimensional, so $p^{+,0}$ is compact if and only if $p^+$ is compact. To prove that $p^+$ is compact, notice that if the coefficients of $\mathfrak{u}$ is in $Exp^+$, then $\mathfrak{u}$ will have exponential increasing Fourier mode. So
\begin{align*}
\sum_{l,m}\sqrt{l^2+m^2}|\hat{u}^+_{l,m}|^2\leq C\|\mathfrak{u}\|^2_{L^2}.
\end{align*}
This inequality implies that: Any converging sequence $\{\mathfrak{u}^k\}$ in $L^2$ will provide a subsequence in $\{p^+(\mathfrak{u}^k)\}$ converging strongly in $l^2$. So $p^+$ is a compact operator.
\end{proof}
By Proposition 4.2 and (4.6), one can check that $F_{T^2,p}$ is Fredholm if and only if $\mathcal{T}_{d^{\pm}}$ is Fredholm. Throughout this paper, we assume the following assumption for $F_{T^2,p}$ in the four-dimensional case:\\

{\bf Assumption}: $\mathcal{T}_{d^{\pm}}$ is Fredholm for any $d^{\pm}$ satisfying $|d^+|^2+|d^-|^2>0$.\\

One can regard this assumption as the 4-dimensional version of Proposition 2.6. So it is conceivable that this assumption is true if we believe that the moduli space of $\mathbb{Z}/2$-harmonic spinors has Kuranishi structure in dimension 4. However, this problem remains open now.\\

The following index theorem is given by Fangyun Yang in \cite[Theorem 1.0.3]{I}.
\begin{theorem}
$dim(ker(p^-))-dim(ker(\mathfrak{p}^{-,0}))=\int \hat{A}(M)+\frac{1}{2}dim(ker(D_{\Sigma}))$.
\end{theorem}

In fact, Fangyun Yang gave a more general version of this index therorem for the $2n$-dimensional manifolds with a embedding codimension 2 submanifold $\Sigma$ . She proved the following formula.
\begin{align*}
dim(ker(p^-))-dim(ker(\mathfrak{p}^{-,0}))=\int \hat{A}(M)+\int \hat{A}(\Sigma)\frac{1-\cosh(\frac{e}{2})}{\sinh(\frac{e}{2})}+\frac{1}{2}dim(ker(D_{\Sigma}))
\end{align*}
where $e$ is the Euler class of the normal bundle of $\Sigma$. However, since the normal bundle of $\Sigma$ is trivial, the middle term will vanish.

\begin{theorem}
Suppose that $\mathcal{T}_{d^{\pm}}|_{Exp^-}$ is Fredholm provided $|d^+|^2+|d^-|^2>0$, then $dim(\mathbb{K}_0)-dim(\mathbb{K}_1)=\int \hat{A}(M).$
\end{theorem}

Under the assumption of Theorem 4.4, we also have the following proposition.
\begin{pro}
Suppose that $\mathcal{T}_{d^{\pm}}|_{Exp^-}$ is Fredholm and $index(\mathcal{T}_{d^{\pm}}|_{Exp^-})= constant$ for all $d^{\pm}$ satisfying $|d^+|^2+|d^-|^2>0$. Then 
\begin{align*}
index(\mathcal{T}_{d^{\pm}}|_{Exp^-})=-\frac{1}{2}dim(ker(D_{\Sigma})).
\end{align*}
\end{pro}

\begin{proof}
By taking $d^+=0, d^-=1$ (or $d^-=e^{i\frac{1}{2}t}$, $d^-=e^{i\frac{1}{2}s}$, $d^-=e^{i\frac{1}{2}t}e^{i\frac{1}{2}s}$ according to $\mathcal{S}^{\pm}_{\Sigma}$), we have
\begin{align*}
\mathcal{T}_{d^{\pm}}|_{Exp^-}(c)=\bar{d}^-\otimes c
\end{align*}
for any $c\in Exp^-$. So it is clearly to see that $ker(\mathcal{T}_{d^{\pm}}|_{Exp^-})=0$. Meanwhile, the cokernel of $\mathcal{T}_{d^{\pm}}|_{Exp^-}$ will be $\bar{d}^-\otimes \Pi^+(ker(D_{\Sigma}))$, where $\Pi^+$ is the projection from $\mathcal{S}_{\Sigma}$ to $\mathcal{S}_{\Sigma}^+$. Therefore we have
\begin{align*}
index(\mathcal{T}_{d^{\pm}}|_{Exp^-})= 0-dim(\Pi^+(ker(D_{\Sigma})))=-\frac{1}{2}dim(ker(D_{\Sigma})).
\end{align*}
\end{proof}

\subsection{Index theorem for $\mathfrak{M}_{T^2,g}$}
With all information above, we are ready to prove Theorem 4.4 now. Firstly, notice that
\begin{align*}
dim(\mathbb{K}_0)-dim(\mathbb{K}_1)=index(\mathcal{T}_{d^{\pm}}\circ B)-dim(ker(D^-|_{L^2_1})).
\end{align*}
So to prove Theorem 4.4, we have to show the following proposition is true. 
\begin{pro}
 $index(\mathcal{T}_{d^{\pm}}\circ B)=\int_M \hat{A}(M)+dim(ker(D^-|_{L^2_1}))$. 
\end{pro}
To begin with, we have to define the 4-dimensional version of $E_\partial$ space.
\begin{definition}
Let 
\begin{align*}
E_{\partial}=\{v\in Dom(D^+)| r^{\frac{1}{2}}v(r,\cdot)\rightharpoonup Z\in (L^2(\mathcal{S}^+_{\Sigma})\otimes e^{-\frac{1}{2}i\theta})\oplus (L^2(\mathcal{S}^-_{\Sigma})&\otimes e^{\frac{1}{2}i\theta})\\
&\mbox{ as } r\rightarrow 0\}.
\end{align*}
and define $\partial(v)=Z$ when the limit exists.
\end{definition}

So for any $v\in Dom(D^+)$, $w\in Dom(D^-)$, we have
\begin{align*}
\int_{M-\Sigma} \langle D^+v,w \rangle+\langle v, D^-w\rangle=2\pi \int_{T^2}\langle \partial(v), e_*\partial(w)\rangle=2\pi \int_{T^2} \langle B(v), e_0 B(w)\rangle.
\end{align*}
It is also easy to check that
\begin{align}
e_0: Exp^{\pm}\rightarrow \mathcal{E}xp^{\mp},\\
e_0: ker(D_{\Sigma})\rightarrow ker(D_{\Sigma})\nonumber
\end{align}
are isomorphisms. Then we have the following 4-dimensional version of Proposition 3.2
\begin{align}
L^2(T^2;S^+\oplus S^-)\cong B(ker(D^+|_{L^2}))\oplus e_0 B(ker(D^-|_{L^2}))\oplus ker(D_{\Sigma})
\end{align}

Now we can prove Proposition 4.6. By using the same argument as we did for 3-dimensional case,
\begin{align*}
coker(p^-)&=p^-(ker(D^+|_{L^2}))=[\pi^-\circ B(ker(D^+|_{L^2}))]^{\perp}\\
          &=[e_0B(ker(D^-|_{L^2}))\oplus ker(D_{\Sigma})]\cap ker(\pi^{+,0})\\
          &=[B(ker(D^-|_{L^2})\oplus ker(D_{\Sigma})]\cap \mathcal{E}xp^{+}\cong ker(\mathfrak{p}^{-,0})/ker(D^-|_{L^2_1}).
\end{align*}
So we have
\begin{align*}
index(\mathcal{T}_{d^{\pm}}\circ B)&=index(\mathcal{T}_{d^{\pm}}|_{Exp^-}\circ p^-)=index(\mathcal{T}_{d^{\pm}}|_{Exp^-})+index(p^-)\\
                         &=index(\mathcal{T}_{d^{\pm}}|_{Exp^-})+ker(p^-)-coker(p^-)\\
                         &=-\frac{1}{2}dim(ker(D_{\Sigma}))+ker(p^-)-ker(\mathfrak{p}^{-,0})+dim(ker(D^-|_{L^2_1}))\\
                         &=-\frac{1}{2}dim(ker(D_{\Sigma}))+\int \hat{A}(M)+\frac{1}{2}dim(ker(D_{\Sigma}))+dim(ker(D^-|_{L^2_1}))\\
                         &=\int \hat{A}(M)+dim(ker(D^-|_{L^2_1})).
\end{align*}
Therefore we prove Proposition 4.6.

\section{Appendix: Proof of Proposition 2.5}
Here we prove the fact that $coker(\mathfrak{L}_p)$ is isomorphic to $coker(\mathcal{T}_{d^{\pm}}\circ B)\oplus ker(D|_{L^2_1})$.

Firstly, we recall that $L^2(M-\Sigma;\mathcal{S}_{g,\Sigma})=range(D|_{L^2_1})\oplus ker(D|_{L^2})$ by Proposition 2.1. Moreover, we have $ker(D|_{L^2})\simeq B(ker(D|_{L^2}))\oplus ker(D|_{L^2_1})$ because any $L^2$-harmonic spinors can be determined by its leading terms and an element in $ker(D|_{L^2_1})$. By definition (2.1) and the fact that $\phi_0\in L^2_1$, we have $coker(\mathfrak{L}_p)=range(\mathfrak{L}_p)^{\perp}\subset ker(D|_{L^2})\simeq B(ker(D|_{L^2}))\oplus ker(D|_{L^2_1})$. So any $u\in coker(\mathfrak{L}_p)$ can be written as a pair $(B(u), v)\in B(ker(D|_{L^2}))\oplus ker(D|_{L^2_1})$. Our goal is to define a 1-1 correspondence mapping $B(u)$ to an element in $coker(\mathcal{T}_{d^{\pm}}\circ B)$.\\

For any $u\in coker(\mathfrak{L}_p)$, we can write $(u^+,u^-)=B(u)$ and derive the following equality
\begin{align}
0= Re\int_{M-\Sigma}\langle u,  \mathfrak{L}_p(\eta,\phi_0) \rangle=Re\int_{S^1} \bar{d^-}\eta u^+-\bar{d}^+\bar{\eta}u^-
\end{align} 
by integration by parts. This equality is true for all $C^1$-maps $\eta:S^1\rightarrow \mathbb{C}$. So we can conclude that
\begin{align*}
d^-\bar{u}^+=\bar{d}^+u^-.
\end{align*}
We define the following $c$ to be the corresponding element in $coker(\mathcal{T}_{d^{\pm}\circ B})$: 
\begin{align*}
c=\frac{\bar{u}^+}{\bar{d}^+}=\frac{u^-}{d^-}.
\end{align*}
$c$ is well-defined because by the definition of $\mathfrak{M}$, we have $\frac{|\psi|(p)}{\mbox{dist}(p,\Sigma)^{\frac{1}{2}}}> 0$ which implies $|d^+|^2+|d^-|^2\neq 0$.\\

Now, we claim that the map $\mathcal{J} :u \rightarrow (c,v)$ is a bijection from $coker(\mathfrak{L}_p)$ to $coker(\mathcal{T}_{d^{\pm}}\circ B)\oplus ker(D|_{L^2_1})$.\\

Before proving this claim, we also have to show that $\mathcal{J}$ is well-defined. In the other words, we have to check that $c$ is in $coker(\mathcal{T}_{d^{\pm}}\circ B)$. To prove this condition, we have to regard $L^2(S^1;\mathbb{C})$ as a real vector space and use the inner product
\begin{align*}
( f,g ) :=Re \int_{S^1} f\bar{g}.
\end{align*}
By using this inner product, for any $(w^+,w^-)=B(w)\in B(ker(D|_{L^2}))$, we have
\begin{align}
(\mathcal{T}_{d^{\pm}}\circ B(w),c)&=Re\int_{S^1}\bar{u}^-w^+-u^+\bar{w}^-=Re\int_{S^1}\bar{u}^-w^+-\bar{u}^+w^-\\
&= Re(\int_{M-\Sigma}\langle w,Du \rangle+\langle Dw,u \rangle)=0.\nonumber
\end{align}
So $c$ is in $coker(\mathcal{T}_{d^{\pm}}\circ B)$.\\

The injectivity of $\mathcal{J}$ is easy to check. So here we only show that $\mathcal{J}$ is surjective. To prove this part, we choose $c'\in coker(\mathcal{T}_{d^{\pm}}\circ B)$ and define $(u'^+,u'^-):=(d^+\bar{c}',d^-c')$. Then $(d^+\bar{c}',d^-c')$ will be perpendicular to $e_0 B(ker(D|_{L^2}))$. By using Proposition 3.2, we have $(d^+\bar{c}',d^-c')\in B(ker(D|_{L^2}))$. So we have $\mathcal{J}$ is surjective.\\

\noindent {\bf Acknowledgement:} The main idea in the first part of this paper was formed during the time that the author had visited Universitat Bielefeld, Germany. The author wants to thank Andriy Haydys, Stefan Bauer and Zvonimir Sviben for their warm hospitality and discussion. He also wants to thank Aleksander Doan, Simon Donaldson, Yi-Jen Lee and especially Clifford Taubes,  for their encouragement. Finally, he wants to thank an anonymous referee who spent a lot of time to make this paper better.\\ 

\bibliographystyle{amsplain}

\end{document}